\documentclass[11pt]{amsart}

\usepackage{amssymb,amsmath,latexsym, amscd}
\input cyracc.def

\newtheorem{theorem}{Theorem}[section]
\newtheorem{lemma}[theorem]{Lemma}
\newtheorem{proposition}[theorem]{Proposition}
\newtheorem{definition}[theorem]{Definition}

\newtheorem{corollary}[theorem]{Corollary}

\def\s{\sigma}

\begin{document}

\title[Free lattice ordered groups and the space of left orderings]{Free lattice ordered groups and the topology on the space of left orderings of a group}

\date{\today}

\author[Adam Clay]{Adam Clay}
\address{Department of Mathematics\\
University of British Columbia \\
Vancouver \\
BC Canada V6T 1Z2} \email{aclay@math.ubc.ca}
\urladdr{http://www.math.ubc.ca/~aclay/} \maketitle

\begin{abstract}For any left orderable group $G$, we recall from work of McCleary that isolated points in the space $LO(G)$ correspond to basic elements in the free lattice ordered group $F(G)$.   We then establish a new connection between the kernels of certain maps in the free lattice ordered group $F(G)$, and the topology on the space of left orderings $LO(G)$.  This connection yields a simple proof that no left orderable group has countably many left orderings.

  When we take $G$ to be the free group $F_n$ of rank $n$, this connection sheds new light on the space of left orderings $LO(F_n)$: by applying a result of Kopytov, we show that there exists a left ordering of the free group whose orbit is dense in the space of left orderings.   From this, we obtain a new proof that $LO(F_n)$ contains no isolated points, and equivalently, a new proof that $F(F_n)$ contains no basic elements.
\end{abstract}

\section{The space of left orderings}

A group $G$ is said to be left orderable if there exists a strict total ordering $<$ of the elements of $G$, such that $g<h$  implies $fg <fh$ for all $f, g, h$ in $G$. Associated to each left ordering of a group $G$ is its positive cone defined by $P = \{ g \in G | g>1 \}$,  elements of the positive cone are said to be positive in the ordering $<$.  The positive cone $P$ of a left ordering $<$ of $G$ satisfies $P \cdot P \subset P$, and $P \sqcup P^{-1} \sqcup \{1 \} = G$;  conversely, any $P \subset G$ satisfying these two properties defines a left invariant total ordering of the elements of $G$, according to the prescription $g<h$ if and only if $g^{-1}h \in P$.   We will denote the left ordering of a group $G$ arising from a positive cone $P$ by $<_P$.

If the positive cone $P$ of a left ordering additionally satisfies $gPg^{-1} = P$ for all $g \in G$, then the associated left ordering of $G$ satisfies $g<_Ph$ implies $gf <_P hf$ and $fg <_P fh$ for all $f, g, h$ in $G$.  In this case, $<_P$ is a bi-ordering of $G$. 

Finally, a left ordering $<_P$ of $G$ is called Conradian if for every pair of positive elements $g, h \in P$, there exists an integer $n$ such that $g <_P h g^n$.  In fact, we can equivalently require that  $g <_P h g^2$ for all pairs of positive elements $g, h$ in $G$ \cite{NF07}.   Observe that all bi-orderings are also Conradian left orderings, but not vice versa.   A group $G$ is Conradian left-orderable if and only if it is locally indicable, meaning that all finitely generated subgroups surject onto the integers \cite{NF07}, \cite{RR02}, \cite{AG99}.

We define the space of left orderings of $G$ as follows.  Denote by $LO(G)$ the set of all subsets $P \subset G$ satisfying  $P \cdot P \subset P$, and $P \sqcup P^{-1} \sqcup \{1 \} = G$, so that we can think of $LO(G)$ as the set of all left orderings of $G$.  Then $LO(G)$ is a subset of the power set of $G$.  Recall that the power set $2^G$ has a natural topology, a subbasis for which is given by the sets
\[ U_g = \{ S \in 2^G | g \in S \}, 
\]
where $g$ ranges over all elements of the group $G$.  Thus, $LO(G)$ inherits a topology, a subbasis for which is given by the sets
\[ U_g = \{ P \subset G | g \in P \},
\]
so that an arbitrary basic open set in $LO(G)$ has the form 
\[ U_{g_1} \cap \cdots \cap U_{g_n} = \{ P \subset G | g_1, \cdots g_n \in P \},\]
where $g_1, \cdots , g_n$ is an arbitrary finite family of elements in $G$.  The basic open set $ U_{g_1} \cap \cdots \cap U_{g_n} \subset LO(G)$ therefore contains all positive cones containing the elements $g_1, \cdots g_n$, corresponding to all those left orderings of $G$ in which $g_1, \cdots g_n$ are positive.

The existence of one-point open sets in $LO(G)$ has been a question of some interest, being raised by Sikora in \cite{AS04}.  A one-point open set in our topology is a set of the form 
\[ \{P \} = U_{g_1} \cap \cdots \cap U_{g_n}, \]
meaning that $P$ is the unique positive cone in $G$ that contains the finite family $g_1, \cdots ,g_n$, in other words, $<_P$ is the unique left ordering of $G$ in which the elements $g_1, \cdots ,g_n $ are positive.

\section{Free lattice-ordered groups}

A group $G$ is said to be lattice-ordered, referred to as an $l$-group, if there exists a \textit{partial} ordering $<$ of the elements of $G$ satisfying:
\begin{enumerate}
\item $g<h$ implies $fg <fh$ and $gf <hf$ for all $f, g, h$ in $G$, and
\item The ordering $<$ admits a lattice structure, that is, every finite set has a least upper bound and a greatest lower bound.
\end{enumerate}
As is standard, we denote the greatest lower bound and the least upper bound of $g_1, \cdots g_n$ by $\bigwedge_{i=1}^n g_i$ and $\bigvee_{i=1}^n g_i$ respectively.   It is an easy computation to show that all lattice-ordered groups must necessarily be distributive lattices.   A homomorphism $h$ of lattice ordered groups from $L_1$ to $L_2$  (often called an $l$-homomorphism) is a map $h:L_1 \rightarrow L_2$ that is simultaneously a group homomorphism and a morphism of lattices, so that $h$ respects the partial ordering $<$, as well as distributing over all finite meets and joins.  The following two examples are standard constructions which are of great importance in what follows.

First, given a totally ordered set $\Omega$ with ordering $<$, the group of all order preserving automorphisms of $\Omega$ forms a lattice ordered group, which we denote by $Aut(\Omega, <)$.  The lattice ordering $\prec$ of $Aut(\Omega, <)$ is defined pointwise:  for any $f, g$ in $Aut(\Omega, <)$,  we declare $f \prec g$ if $f(x) < g(x)$ for all $x \in \Omega$.

As a second example, let $\{ L_i \}_{i \in I}$, be an arbitrary collection of lattice ordered groups, with $L_i$ having the lattice ordering $\prec_i$.  We can form a new lattice ordered group $L$ by setting
\[ L = \prod_{i \in I} L_i \]
and for any $x, y$ in $L$ we declare $x \prec y$ if $\pi_i(x) <_i \pi_i(y)$ for all $i \in I$.  Here, $\pi_i : L \rightarrow L_i$ is projection onto the $i$-th component in the product.

We are now ready to introduce the main object of concern in this paper.  Let $G$ be a left orderable group.  A lattice ordered group $F(G)$ is said to be the free lattice ordered group over $G$ if $F(G)$ satisfies:
\begin{enumerate}
\item  There exists an injective homomorphism $i: G \rightarrow F(G)$ such that $i(G)$ generates $F(G)$ as an $l$-group. 
\item  For any lattice ordered group $L$ and any homomorphism of groups $\phi : G \rightarrow L$ there exists a unique $l$-homomorphism $\bar{\phi} : F(G) \rightarrow L$ such that $\bar{\phi} \circ i = \phi$.
\end{enumerate}
Obviously any such group is unique up to $l$-isomorphism.

We present a construction of the free lattice ordered group $F(G)$ due to Conrad \cite{PC70}, in the case that $G$ is a left orderable group.

For each positive cone $P \in LO(G)$, the group $Aut(G, <_P)$ is a lattice ordered group, whose partial ordering we will denote by $\prec_P$.  We may embed the group $G$ into $Aut(G, <_P)$ by sending each $g \in G$ to the order-preserving automorphism of the totally ordered set $(G, <_P)$ defined by left-multiplication by $g$.  Denote this order preserving automorphism by $g_P$, so that $g_P \in Aut(G, <_P)$ has action $g_P(h) = gh$.

Define a map 
\[ i: G \rightarrow \prod_{P \in LO(G)} Aut(G, <_P), \]
according to the rule $\pi_P(i(g)) = g_P$, where $\pi_P:  \prod_{P \in LO(G)} Aut(G, <_P) \rightarrow Aut(G, <_P)$ is projection.  Thus, on each factor in the product, $g \in G$ acts by left multiplication.

Denote by $F(G)$ the smallest lattice-ordered subgroup of \[ \prod_{P \in LO(G)} Aut(G, <_P)\] containing the set $i(G)$.  Then $F(G)$, together with the map $i$ defined above, is the free lattice ordered group over the left-orderable group $G$.  The lattice ordering of $F(G)$ will be denoted by $\prec$.

Essential in showing that this construction produces a group satisfying the required universal property is the following proposition, due to Conrad \cite{PC70}.
\begin{proposition}
\label{prop:con} Let $G$ be a left orderable group, and
suppose that $x$ is any non-identity element of $F(G)$.  Then there exists $P \in LO(G)$ such that $\pi_P(x) \neq 1$. 
\end{proposition}
\begin{corollary}
\label{cor:1}
Let $x \in F(G)$ be any element satisfying $\pi_P(x)(1) = 1$ for all $P$ in $LO(G)$.  Then $x=1$.
\end{corollary}
\begin{proof}
Suppose that $x$ is a non-identity element of $F(G)$, from Conrad's proposition there exists $P$ such that $\pi_P(x)(h) \neq h $ for some $h$ in $G$.  It follows that $\pi_Q(x)(1) \neq 1$, for $Q = h^{-1}Ph$.
\end{proof}

From this point forward we will simplify our notation by writing $g \in F(G)$ in place of $i(g) \in F(G)$, for any element $g$ of $G$.  Thus, any element of $F(G)$ can be (non-uniquely) written in the form $\bigvee_{i \in I} \bigwedge_{j \in J} g_{ij}$ for suitable $g_{ij} \in G$.  The map $\pi_P:F(G) \rightarrow Aut(G, <_P)$ sends such an element to the order-preserving map whose action on an element $h \in G$ is defined according to the rule:
\[\pi_P(\bigvee_{i \in I} \bigwedge_{j \in J} g_{ij})(h) = max_{i \in I}^{P} min_{j \in J}^P \{ g_{ij}h \}
.\]
Here, the superscript $P$ appearing above $max$ and $min$ indicates that the $max$ and $min$ are taken relative to the total ordering $<_P$ of $G$.

\section{Free lattice ordered groups and the topology on the space of left orderings}

In this section, we establish several connections between the topology of $LO(G)$ and the structure of the group $F(G)$.  We begin with a generalization of a known result, which was originally proven in the case of $F(F_n)$, the free lattice-ordered group over a free group \cite{MA86}.  We observe, however, that the same result holds for any left-orderable group $G$.

Recall that an element $x$ in a lattice-ordered group $L$ is said to be a basic element if the set $ \{ y \in L | e \leq y \leq x \}$ is totally ordered by the restriction of the lattice ordering.  

\begin{lemma}\cite{MA86}
\label{lem:basic}
Suppose that $1 \prec x$ is an element of $F(G)$.  Then $x$ is a basic element if and only if there exists a unique left ordering $<_P$ of $G$ such that $\pi_P(x)(1) >_P 1$.
\end{lemma}
\begin{proof}
Suppose that $<_P$ is the unique left ordering of $G$ for which $\pi_P(x)(1) >_P 1$, and suppose that $y_1, y_2$ are two distinct elements of $ F(G)$ that satisfy $1 \prec y_i \prec x$.  Without loss of generality, we may assume that $\pi_P(y_1)(1) \leq_P \pi_P(y_2)(1)$, and hence $1 \leq_P \pi_P(y_1^{-1} y_2)(1)$.   Therefore, considering the element $y_1^{-1}y_2 \wedge 1 \in F(G)$, we compute that $\pi_P(y_1^{-1}y_2 \wedge 1)(1) = min\{ \pi_P(y_1^{-1}y_2)(1), 1 \} = 1$.

Now in any left ordering $<_Q$ with $Q \neq P$, we have $1 \leq_Q \pi_Q(y_i)(1) \leq_Q \pi_Q(x)(1) \leq_Q 1$, where the final inequality follows from our assumption that $<_P$ is the unique left ordering with $\pi_P(x)(1) >_P 1$.   

Therefore, $\pi_P(y_1^{-1}y_2 \wedge 1)(1) = 1$ for all orderings $P$ of $G$.   It follows that $y_1^{-1}y_2 \wedge 1 = 1$ by Corollary \ref{cor:1}, and hence $1 \prec y_1^{-1}y_2$, and $y_2 \prec y_1$ as desired.

On the other hand, suppose that $x$ is a (positive) basic element, and suppose that $P$ and $Q$ are distinct positive cones such that $\pi_P(x)(1) >_P 1$ and  $\pi_Q(x)(1) >_Q 1$.  Choose an element $h$ of $G$ such that $h >_P 1$ and $h^{-1} >_Q 1$.  Then the elements $y_1 = (x \wedge h) \vee 1$ and $y_2 = (x \wedge h^{-1}) \vee 1$  satisfy $1 \prec y_i \prec x$, yet are not comparable in the partial ordering $\prec$ of $F(G)$.  This follows from computing
\[\pi_P(y_1)(1) = max\{ min\{ \pi_P(x)(1), h \}, 1\} >_P 1, 
\] 
and
\[ \pi_P(y_2)(1) = max \{ min\{ \pi_P(x)(1), h^{-1} \}, 1 \} = 1, \]
while $\pi_Q(y_1)(1) = 1$ and $\pi_Q(y_2)(1) >_Q 1$.

\end{proof}

\begin{theorem} \cite{MA86}
\label{th:basic}
Let $G$ be a left orderable group.  Then $F(G)$ contains a basic element if and only if $LO(G)$ contains an isolated point.
 \end{theorem}
\begin{proof}
Suppose that 
\[ \{P \} = U_{g_1} \cap \cdots \cap U_{g_n} \] 
is an isolated point in $LO(G)$.  Then $P$ is the unique left ordering for which $1 <_P g_j$ for all $j$, and is therefore the unique ordering for which $\pi_P(g_j)(1) >_P 1$ for all $j$.  Therefore, $<_P$ is the unique left ordering for which $\pi_P(\bigwedge_{j=1}^n g_j)(1) >_P 1$, and hence $<_P$ is the unique left ordering for which  $\pi_P((\bigwedge_{j=1}^n g_j) \vee 1)(1) >_P 1$.  It follows that $(\bigwedge_{j=1}^n g_j) \vee 1$ is a basic element of $F(G)$, by Lemma \ref{lem:basic}.

Conversely, if $\bigvee_{i \in I}^m \bigwedge_{j \in J}^n g_{ij}$ is a basic element in $F(G)$, then there is a unique left ordering $<_P$ of $G$ such that $\pi_P(\bigvee_{i \in I}^m \bigwedge_{j \in J}^n g_{ij})(1) >_P 1$.   Therefore, for some index $i$, $\pi_P(\bigwedge_{j \in J} g_{ij})(1) = max_{j \in J} \{g_{ij} \} >_P 1$, and $<_P$ is the unique left ordering of $G$ for which this inequality holds, for our chosen index $i$.  In other words, $<_P$ is the unique left ordering in which all the elements $g_{ij}$ are positive for our chosen index $i$, so that 
\[ \{ P \} = \bigcap_{j \in J} U_{g_{ij}} \]
 is an isolated point in $LO(G)$.
\end{proof}

Recall that for any left-orderable group $G$, the space $LO(G)$ comes equipped with a $G$-action by conjugation, which is an action by homeomorphisms.  Set
\[ Orb_G(P) = \{ gPg^{-1} | g \in G \}
,\]
and we denote the closure of each such orbit by $\overline{Orb_G(P)}$.  The group action on the space of left orderings $LO(G)$ has been used to great success in investigating the structure of the space of left orderings, see \cite{NF07}, \cite{AC08}, \cite{AR09}, \cite{NW09}.  However, in most applications, one often asks if a positive cone $P$ is an accumulation point of its own conjugates in order to show that $P$ is not an isolated point.  Perhaps more useful would be a way of answering the question:  When is $Q$ an accumulation point of the conjugates of some different positive cone $P$, that is, $Q  \in \overline{Orb_G(P)}$?  We find that this question is equivalent to an algebraic question about the lattice-ordered free group $F(G)$.

\begin{theorem}
\label{th:acc}
Let $G$ be a left orderable group, and let $P, Q \in LO(G)$ be given.  Then $Q \in \overline{Orb_G(P)}$ if and only if $ker( \pi_P) \subset ker( \pi_Q)$.
\end{theorem}
\begin{proof}
Suppose that $ker( \pi_P) \subset ker( \pi_Q)$, and that $Q$ lies in the basic open set $\bigcap_{j=1}^n U_{g_j}$.  We must show that some conjugate of $P$ lies in this open set as well.

Consider the element $(\bigwedge_{j=1}^n g_j) \vee 1$ in $F(G)$.  As $Q \in \bigcap_{j=1}^n U_{g_j}$, we know that $g_j >_Q 1$ for all $j$, and hence we find that
\[ \pi_Q((\bigwedge_{j=1}^n g_j) \vee 1)(1) = \max\{ \min\{g_j \}, 1\} = \min\{g_j\} >_Q 1 .
\]
Therefore, $(\bigwedge_{j=1}^n g_j) \vee 1$ is not in the kernel of the map $\pi_Q$, and so from our assumption it is not in the kernel of the map $\pi_P$.  Thus, there exists $h \in G$ such that $\pi_P((\bigwedge_{j=1}^n g_j) \vee 1)(h) \neq h$, and we compute
\[ \pi_P((\bigwedge_{j=1}^n g_j) \vee 1)(h) = \max\{ \min\{g_jh \}, h\} >_P h,\]
so that $g_jh >_P h$ for $j=1, \cdots , n$.  Therefore, $h^{-1}g_jh>_P1$, and hence $h^{-1}g_jh \in P$ for all $j$.  This is equivalent to $g_j \in hPh^{-1}$ for all $j$, or $hPh^{-1} \in \bigcap_{j=1}^n U_{g_j}$, so that $Q$ is an accumulation point of the orbit of $P$. 

On the other hand, suppose that $Q \in \overline{Orb_G(P)}$, and let $\bigvee_{i \in I} \bigwedge_{j \in J} g_{ij}$ be any element of $F(G)$ such that $\pi_Q(\bigvee_{i \in I} \bigwedge_{j \in J} g_{ij}) \neq 1$.  There are two cases to consider, in order to show that $\bigvee_{i \in I} \bigwedge_{j \in J} g_{ij} \notin ker( \pi_P)$.

\textbf{Case 1.} There exists $i$ such that $\pi_Q(\bigwedge_{j \in J} g_{ij})(h) >_Q h$ for some $h \in G$.  Then $\min_j{g_{ij}h} >_Q h$, and therefore $h^{-1}g_{ij}h \in Q$ for all $j$, hence $ Q \in \bigcap_{j \in J} U_{h^{-1}g_{ij}h}$.  By assumption, we can choose $f \in G$ such that $fPf^{-1} \in \bigcap_{j \in J} U_{h^{-1}g_{ij}h}$, so that $h^{-1}g_{ij}h \in fPf^{-1}$ for all $j$.  In other words, $f^{-1}h^{-1}g_{ij}hf >_P 1$ for all $j$, so that $g_{ij} hf >_P hf$ for all $j$.  

Now, we may compute 
\[ \pi_P(\bigvee_{i \in I} \bigwedge_{j \in J} g_{ij})(hf) >_P \pi_P(\bigwedge_{j \in J} g_{ij})(hf) = \min_j\{ g_{ij}hf \} >_P hf.\]
We conclude that $\bigvee_{i \in I} \bigwedge_{j \in J} g_{ij} \notin ker( \pi_P)$.

\textbf{Case 2.} For all $i$, $\pi_Q(\bigwedge_{j \in J} g_{ij})(h) \leq_Q h$ for all $h \in G$. Then in particular, we may choose $h \in G$ such that $\pi_Q(\bigwedge_{j \in J} g_{ij})(h) <_Q h$ for every $i \in I$ (strict inequality), since the image of  $\bigvee_{i \in I} \bigwedge_{j \in J} g_{ij}$ must act nontrivially on $(G, <_Q)$.

Now with $h$ as above, we observe that for every index $i$ there exists an index $j$ such that $g_{ij}h <_Q h$.  Thus, for every $i$, we may choose $h_i$ from the set of elements $\{g_{ij} \}_{j \in J}$, so that $h_i$ satisfies $h_ih <_Q h$.  Then each $h_i$ satisfies $h^{-1} h_i h <_Q 1$, so that $Q \in \bigcap_{i \in I} U_{h^{-1}h_i^{-1}h}$.  We may therefore choose $f \in G$ so that $fPf^{-1} \in \bigcap_{i \in I} U_{h^{-1}h_i^{-1}h}$, in other words, $f^{-1} h^{-1} h_i^{-1} h f >_P 1$ for all $i \in I$.  It follows that $f^{-1} h^{-1} h_i h f <_P 1$ and so $h_ih f <_P hf$ for all $i \in I$. 

Thus, we find that for every $i \in I$, 
\[ \pi_P(\bigwedge_{j \in J} g_{ij})(hf) = \min_j \{ g_{ij} hf \} \leq_P h_ihf <_P hf,\]
where the inequality $\min_j\{ g_{ij} hf \} \leq_P h_ihf$ follows from the fact that $h_i$ lies in the set $\{g_{ij} \}_{j \in J}$.

Thus, when we take a (finite) maximum over all $i$, we compute that  $\pi_P(\bigvee_{i \in I} \bigwedge_{j \in J} g_{ij})(hf) = \min_j\{ g_{ij} hf \}$ for some $i$, and hence 
\[ \pi_P(\bigvee_{i \in I} \bigwedge_{j \in J} g_{ij})(hf) = \min_j\{ g_{ij} hf \}  \leq_P h_ihf <_P hf.
\]
It follows that $ \pi_P(\bigvee_{i \in I} \bigwedge_{j \in J} g_{ij})$ is nontrivial, and the claim is proven.
\end{proof}

\begin{corollary}
\label{cor:dense1}
For a given positive cone $P$ in a left orderable group $G$, $\overline{Orb_G(P)} = LO(G)$ if and only if $\pi_P : F(G) \rightarrow Aut(G, <_P)$ is injective.
\end{corollary}
\begin{proof}
It is clear that if $\pi_P : F(G) \rightarrow Aut(G, <_P)$ has trivial kernel, then $\overline{Orb_G(P)} = LO(G)$ by Theorem \ref{th:acc}.

For the other direction, suppose that $\overline{Orb_G(P)} = LO(G)$.  From Theorem \ref{th:acc} we deduce the containment $ker(\pi_P) \subset \bigcap_{Q \in LO(G)} ker(\pi_Q)$.  However, from Proposition \ref{prop:con} we find that $\bigcap_{Q \in LO(G)} ker(\pi_Q) = \{1 \}$, so that $ker(\pi_P) = \{1 \}$.
\end{proof}

Thus, for a given left orderable group $G$, injectivity of some map $\pi_P : F(G) \rightarrow Aut(G, <_P)$ for some ordering $P \in LO(G)$ tells us a great deal about the structure of $LO(G)$. 

\begin{proposition}
\label{prop:noiso}
Let $G$ be a left orderable group, and suppose that there exists $P \in LO(G)$ such that $\pi_P$ is injective.  Then $LO(G)$ contains no isolated points.
\end{proposition}
\begin{proof}
If the map $\pi_P$ is injective, then we know that we may write $LO(G) = \overline{Orb_G(P)}$, and so only those points in $Orb_G(P)$ itself are possibly isolated in $LO(G)$, which can only happen if $P$ itself is isolated.

Supposing that $P$ is an isolated point, it follows that $P^{-1}$ is also an isolated point in $LO(G)$, and hence $P^{-1} \in Orb_G(P)$; so we may write $P^{-1} = gPg^{-1}$ for some $g \in G$, with $g$ different from the identity.  This is impossible, for supposing $g \in P$ yields (upon conjugation by $g$) $g \in gPg^{-1} = P^{-1}$.  Similarly, $g \in P^{-1}$ is impossible, so that $P$ is not isolated. 
\end{proof}

From Theorem \ref{th:acc}, we now have a bijection between certain normal subgroups of $F(G)$ and certain closed subsets of $LO(G)$.   Specifically, if $K$ is the kernel of the map $\pi_P$, we can associate to $K$ the closed set $\overline{Orb_G(P)}$.  Note that if $\pi_Q$ is some other map with kernel $K$, then $\overline{Orb_G(Q)}= \overline{Orb_G(P)}$, so that the closed set associated to $K$ is well-defined.   Inclusion of kernels $ker(\pi_P) \subset \ker(\pi_Q)$ yields a reverse inclusion of associated subsets, $\overline{Orb_G(Q)} \subset \overline{Orb_G(P)}$.

It is well known that the space $LO(G)$ is compact \cite{AS04}, and this fact has been used to great success: In \cite{PL06}, compactness is the key ingredient in showing that no group has countably many left orderings (which has been proven again recently in \cite{AR09}), and in \cite{DWM06}, compactness is used to show that a left orderable group is amenable if and only if it is locally indicable.  In our present setting, compactness of the space of left orderings yields the following:

If $G$ is a left orderable group, let $S_G$ denote the set of all normal subgroups of $F(G)$ that occur as the kernel of some map $\pi_P : F(G) \rightarrow Aut(G, <_P)$, where $P$ ranges over all positive cones in $LO(G)$ for some left orderable group $G$.  The set $S_G$ is partially ordered by inclusion.

\begin{proposition}
\label{prop:up}
Every chain in $S_G$ has an upper bound, in particular, $S_G$ has a maximal element.
\end{proposition}
\begin{proof}
Let $T$ be a subset of $LO(G)$ such that $\{ \ker(\pi_P) \}_{P \in T}$ is a totally ordered subset of $S_G$.  Observe that $\overline{Orb_G(P)} \subset \overline{Orb_G(Q)}$ if and only if $ker(\pi_Q) \subset ker(\pi_P)$, and thus $\{\overline{Orb_G(P)}\}_{P \in T}$ is a nested collection of closed subsets of $LO(G)$.  In particular, this nested collection of sets has the finite intersection property.  For if $P_1, \cdots ,P_n$ is some finite subset of $T$, upon renumbering if necessary, we may assume that $\overline{Orb_G(P_1)} \subset \cdots \subset \overline{Orb_G(P_n)}$, from which it is obvious that $P_1$ is contained in their intersection.

Thus, the intersection 
\[\bigcap_{P \in T} \overline{Orb_G(P)} \]
is nonempty, as $LO(G)$ is compact.   Choosing any positive cone $R$ from this intersection yields a closed set $\overline{Orb_G(R)}$ that lies in $\overline{Orb_G(P)}$ for every $P$ in $T$, and hence $\ker(\pi_P) \subset \ker(\pi_R)$ for every $P$ in $T$.  It now follows from Zorn's lemma that $S_G$ contains a maximal element.
\end{proof}

The following is a standard definition in the theory of dynamical systems.
\begin{definition}
A nonempty set $U$ in $LO(G)$ is said to be a minimal invariant set if $U$ is closed and $G$-invariant, and for every closed $G$-invariant set $V$ in $LO(G)$, $U \cap V \neq \emptyset$ implies $U \subset V$.
\end{definition}

The equivalence of (1) and (2) in the following proposition is a standard result from the theory of dynamical systems (\cite{JV93}, pp. 69-70).

\begin{proposition}
\label{prop:sm}
  For any nonempty closed subset $U$ of $LO(G)$, the following are equivalent:
\begin{enumerate}
\item $U$ is a minimal invariant set
\item for every $P \in U$, $U = \overline{Orb_G(P)}$
\item $U =\overline{Orb_G(P)}$ for some $P \in LO(G)$ whose kernel is maximal in $S_G$.
\end{enumerate}
\end{proposition}
\begin{proof}
(1) if and only if (2).  Suppose that $U$ is a minimal invariant set, and let $P \in U$ be given.  Then $\overline{Orb_G(P)} \subset U$, since $U$ is closed and $G$-invariant.  Since $U$ is small, this implies $U \subset \overline{Orb_G(P)}$, and so $U = \overline{Orb_G(P)}$.  Conversely, suppose that (2) is satisfied and let $V$ be some other closed, $G$-invariant set such that $U \cap V$ is nonempty.  Choose $Q \in U \cap V$, and observe that $U =  \overline{Orb_G(Q)} \subset V$, since $V$ is closed and $G$-invariant.  Therefore $U$ is a minimal invariant set.

(2) if and only if (3).  Suppose property (2) holds, and let $P \in U$ be given, and suppose that $ker(\pi_P) \subset ker(\pi_Q)$ for some $Q \in LO(G)$.  Then by Theorem \ref{th:acc} $Q \in \overline{Orb_G(P)}=U$, and hence, by condition (2),  $\overline{Orb_G(Q)} = U = \overline{Orb_G(P)}$.   It follows that $ker(\pi_P) = ker(\pi_Q)$ is maximal.  Conversely, suppose (3) and let $P \in U$ be given.   Then for any other $Q$ in $U = \overline{Orb_G(P)}$, we have $\ker(\pi_P) \subset ker(\pi_Q)$ by Theorem \ref{th:acc}.  Since $ker(\pi_P)$ is maximal, this gives $ker(\pi_Q) = ker(\pi_P)$ and (2) follows.
\end{proof}

We can now see that Proposition \ref{prop:up} mirrors a standard proof of the existence of minimal invariant sets, see for example \cite{JV93} Theorem 3.12.  It is also clear from the above characterization that $ker(\pi_P)$ is maximal if $P$ is the positive cone of a bi-ordering (so $\overline{Orb_G(P)} = \{ P \}$) , or if $<_P$ is a bi-ordering when restricted to some finite index subgroup $H \subset G$ (for then $\overline{Orb_G(P)}$ is finite).  

We may apply the notion of minimal invariant sets to provide yet another proof that no left orderable group has countably infinitely many left orderings.

\begin{proposition} Let $G$ be a left orderable group, and let $U$ be a minimal invariant subset of $LO(G)$.  Then $U$ is finite, or uncountable.
\end{proposition}
\begin{proof}Suppose that $U$ is not finite.  If $U$ is infinite and contains no isolated points, then $U$ is uncountable, as $U$ is a compact Hausdorff space [Theorem 2-80, \cite{HY61}].  Thus, suppose that $U$ contains an isolated point, and we will arrive at a contradiction.  Note that in this context, ``isolated point'' means isolated in $U$, not isolated in $LO(G)$.

Choose $P$ in $U$.  By (2) of Proposition \ref{prop:sm}, $U = \overline{Orb_G(P)}$, and it follows that $P$ itself must be an isolated point (and hence every point in $Orb_G(P)$ is isolated).  Since $U$ is compact, the set of conjugates $Orb_G(P)$ must accumulate on some $Q \in U$, moreover, $Q$ does not lie in $Orb_G(P)$, since $Q$ is not isolated.  We again apply (2) of Proposition \ref{prop:sm} to find that $U = \overline{Orb_G(Q)}$, and it follows that $P$ cannot be an isolated point, a contradiction.
\end{proof}

\begin{corollary}
For any group $G$, $LO(G)$ is either finite or uncountable.
\end{corollary}
\begin{proof}
Given any left orderable group $G$, by Proposition \ref{prop:up}, there exists $P \in LO(G)$ whose kernel is maximal in $S_G$.  Correspondingly, the set $\overline{Orb_G(P)}$ is a minimal invariant set, and hence it is either finite, or uncountable. 

Assuming $\overline{Orb_G(P)}$ is finite, we must have that $Orb_G(P)$ is finite, and hence the stabilizer $Stab_G(P)$ is a finite index subgroup of $G$ that is bi-ordered by the restriction of the left ordering $<_P$.  It follows that $G$ is locally indicable \cite{RR02}, and hence has uncountably many left orderings, by \cite{ZV97}.
\end{proof}

It should be noted that this proof is very similar to the proof of Peter Linnell, given in \cite{PL06}.  The crucial difference in our proof is that some difficult topological arguments have been replaced by an application of Zorn's lemma.

It does not appear that a clean topological statement will characterize precisely those closed sets that occur as $\overline{Orb_G(P)}$ for some ordering whose associated kernel is minimal in $S_G$.  We do, however, have the following observation.

\begin{proposition}
Let $G$ be a left orderable group, and let $P$ in $LO(G)$ be an isolated point.  Then $ker(\pi_P)$ is minimal in $S_G$.
\end{proposition}
\begin{proof}
Suppose that $ker(\pi_Q) \subset ker(\pi_P)$ for some $Q \in LO(G)$.  Then $P \in \overline{Orb_G(Q)}$ by Theorem \ref{th:acc}, but $P$ is isolated, so $P \in Orb_G(Q)$. Therefore $Q$ is conjugate to $P$, and so $Q \in  \overline{Orb_G(P)}$, and  $ker(\pi_P) \subset ker(\pi_Q)$ by  Theorem \ref{th:acc}.  Thus $ker(\pi_Q)$ is equal to $ker(\pi_P)$, so that  $ker(\pi_P)$ is minimal.
\end{proof}

Not every minimal kernel in $S_G$ corresponds to an isolated point in $LO(G)$.   In the next section, we will see that with $G=F_n$, the free group on $n$ generators, $S_{F_n}$ contains a minimal kernel (\cite{KO79}, \cite{KO83}) that does not correspond to an isolated point.

\section{Examples}

\subsection{The free group}
It now follows easily from work of Kopytov that $LO(F_n)$ contains a dense orbit under the conjugation action by $F_n$.  This case appears to be the first known example of a left ordering (of any group) whose orbit is dense in the space of left orderings.
\begin{corollary}
\label{cor:iso}
Let $F_n$ denote the free group on $n>1$ generators.  There exists $P$ such that  $\overline{Orb_{F_n}(P)} = LO(F_n)$.
\end{corollary}
\begin{proof}
There exists a left ordering of $F_n$ with positive cone $P$ such that $\pi_P : F(G) \rightarrow Aut(G, <_P)$ is injective (\cite{KO79}, \cite{KO83}), and we apply Corollary \ref{cor:dense1}.
\end{proof}

We may apply Proposition \ref{prop:noiso} and Theorem \ref{th:basic} to arrive at a new proof of the following corollary (see \cite{SM85}, \cite{NF07} for alternate proofs).

\begin{corollary}
The space $LO(F_n)$ has no isolated points for $n >1$.   Equivalently, $F(F_n)$ has no basic elements for $n>1$.
\end{corollary}

Recall that the Conradian soul of a left ordering is the largest convex subgroup on which the restriction ordering is Conradian \cite{NF07}.  It was shown in \cite{SM85} and \cite{MD93} that the construction of Kopytov can be improved so that the map $\pi_P$ is injective, and the left ordering $<_P$ of $F_n$ has no convex subgroups.  Consequently, the Conradian soul of the ordering $<_P$ must be trivial, and so $P$ is an accumulation point of its own conjugates in $LO(F_n)$ \cite{NF07}, \cite{AC08}.   

It is worth noting that McCleary's construction in \cite{SM85} of a faithful $o-2$ transitive action of $F(F_n)$ on some linearly ordered set is much stronger than is needed to conclude that $F(F_n)$ has no basic elements; yet \cite{SM85} appears to contain the first proof of this fact appearing in the literature.

\subsection{Left orderable groups with all left orderings Conradian}

In the case that all left orderings of a given left orderable group $G$ are Conradian, we may highlight two cases of interest.  First we will observe that no finitely generated group $G$, all of whose left orderings are Conradian, can have a dense orbit in $LO(G)$.  Second, we will show that if we allow the group to be infinitely generated, then it may be the case that every orbit in $LO(G)$ is dense.  Recall that, for example, every left ordering of a torsion free locally nilpotent group is Conradian \cite{AJC72}, \cite{MR77}.

An element $g$ in a left ordered group $G$ with ordering $<$ is cofinal if for every $h$ in $G$, there exists $n \in \mathbb{Z}$ such that $g^{-n} < h < g^n$.  

\begin{proposition}
\label{prop:cof}
Suppose that $G$ is a left orderable group, and suppose that $g$ is cofinal in $<_P$, for some $P \in LO(G)$.  Then $P^{-1}$ is not in $\overline{Orb_G(P)}$.
\end{proposition}
\begin{proof}
Suppose that $g>_P1$ is cofinal, and let $1<_P h \in G$ be given.  Choose an integer $n$ so that $h <_P g^n$, and observe that $1<_P h^{-1}g^n$, and hence $1<_Ph^{-1}g^nh$, since $h$ is positive.  Therefore $1<_Ph^{-1}gh$, and so $h <_P gh$.  On the other hand, if $h$ is negative, we similarly conclude that $h <_P gh$, and thus $\pi_P(g \wedge 1)(h) = h$ for all $h \in G$, so that $g \wedge 1 \in ker(\pi_P)$.  On the other hand, in the reverse ordering with positive cone $P^{-1}$, we find $h >_{P^{-1}} gh$ for all $h \in G$, so that $\pi_{P^{-1}}(g \wedge 1)(h) <_{P^{-1}} h$ for all $h \in G$, so that $g \wedge 1 \neq 1$ in $F(G)$.  Thus the map $\pi_P$ is not injective, and in particular, $ker(\pi_P)$ is not contained in $\ker(\pi_{P^{-1}})$.  By Theorem \ref{th:acc}, $P^{-1}$ is not in $\overline{Orb_G(P)}$.
\end{proof}

\begin{corollary}
Let $G$ be a finitely generated group, all of whose left orderings are Conradian.  Then $LO(G)$ does not contain a dense orbit.
\end{corollary}
\begin{proof}
Let $P$ be any positive cone in a finitely generated group $G$.  Since the associated left ordering $<_P$ is Conradian, there exists a convex subgroup $C \subset G$ such that $G/C$ is abelian, and the induced ordering on $G/C$ is Archimedean.  Every element in $G \setminus C$ is then cofinal in the left ordering $<_P$, and the claim follows.
\end{proof}

Next, we consider the group 
\[ T_{\infty} =  \langle x_i, i \in \mathbb{N} : x_{i+1}x_i x_{i+1} = x_i^{-1}, x_ix_j = x_j x_i \mbox{ for $|i-j|>1$} \rangle. \]

Note that each subgroup $T_n = \langle x_1, \cdots ,x_n \rangle $ has only $2^n$ left orderings, as it is one of the Tararin groups, as described in \cite{KME96}, Theorem 5.2.1.  Moreover, the convex subgroups of any left ordering of $T_n$ are precisely $T_i \subset T_n$ for $i \leq n$, and so the convex subgroups of $T_{\infty}$ are exactly $T_i$ for $i \in \mathbb{N}$, for any left ordering of $T_{\infty}$.  The orderings of $T_{\infty}$ are all Conradian, with convex jumps $T_{i+1} / T_i \cong \mathbb{Z}$.   Given any positive cone $P \in LO(T_{\infty})$, it is therefore determined by the signs of the generators $x_i$, which we record in a sequence
\[\varepsilon = ( \pm1, \pm1, \cdots )
\]
writing $+1$ in the $i$-th position if $x_i$ is in $P$, and $-1$ otherwise.  We then write $P = P_{\varepsilon}$.  With this notation, we observe that $x_{i+1}  P_{\varepsilon} x_{i+1}^{-1} =  P_{\varepsilon'}$, where $\varepsilon'$ differs from $\varepsilon$ only in the sign of the $i$-th entry (this follows from the defining relations of the group $T_{\infty}$, and corresponds to the idea of ``flipping'' the ordering on the $i$-th convex jump).

\begin{proposition}
Let $P \in LO(T_{\infty})$ be any positive cone.  Then $LO(T_{\infty}) = \overline{Orb_{T_{\infty}}(P)}$.
\end{proposition}
\begin{proof}
Let $P_{\varepsilon_1}$ and $P_{\varepsilon_2}$ be two positive cones in $LO(T_{\infty})$.  It is enough to show that for every $n \in \mathbb{N}$, there exists $g \in T_{\infty}$ such that $gP_{\varepsilon_1}g^{-1} \cap T_n = P_{\varepsilon_2} \cap T_n$, so that the associated left orderings agree upon restriction to the subgroup $T_n$. 

The proof is a simple induction.  First, note that there exists a conjugate of  $P_{\varepsilon_1}$ that agrees with $P_{\varepsilon_2}$ upon restriction to the subgroup $T_1 = \langle x_1 \rangle$: if $\varepsilon_1$ and $\varepsilon_2$ agree in the first entry, then $P_{\varepsilon_1}$ and $P_{\varepsilon_2}$ agree on $T_1$, whereas if $\varepsilon_1$ and $\varepsilon_2$ disagree in the first entry, then $x_2P_{\varepsilon_1}x_2^{-1}$ and $P_{\varepsilon_2}$ agree on $T_1$.

 For induction, suppose that $Q$ is a conjugate of $P_{\varepsilon_1}$ with associated sequence $\varepsilon_1'$ first differing from $\varepsilon_2$ in the $n$-th entry, so that $Q \cap T_{n-1} = P_{\varepsilon_2} \cap T_{n-1}$.  Then $x_{n+1} Q x_{n+1}^{-1}$ will have associated sequence $\varepsilon_1''$ that agrees with $\varepsilon_2$ in the sign of the $n$-th entry.  Thus the conjugate $x_{n+1} Q x_{n+1}^{-1}$ of $P_{\varepsilon_1}$ agrees with $P_{\varepsilon_2}$ upon restriction to the subgroup $T_n$, and the result follows by induction.
\end{proof}

\subsection{The braid groups}
The braid groups $B_n, n>2$ also provide an interesting class of examples, as their spaces of left orderings are known to contain isolated points \cite{DD04}, \cite{NF07}, \cite{AC08}.   

\begin{proposition}
\label{prop:braid}
For $n>2$, the space $LO(B_n)$ contains no dense orbit.
\end{proposition}
\begin{proof}
This follows immediately from Proposition \ref{prop:noiso}, in light of the fact that $LO(B_n)$ contains isolated points.
\end{proof} 

As an alternative way of proving Proposition \ref{prop:braid}, recall that the center of the braid group $B_n$, for $n>1$, is infinite cyclic with generator $\Delta_n^2$.  Here, 
$$\Delta_k := (\s_{k-1} \s_{k-2} \cdots \s_{1})(\s_{k-1} \s_{k-2} \cdots \s_{2}) \cdots (\s_{k-1}\s_{k-2})(\s_{k-1})$$
is the Garside half-twist.  It is well known that $\Delta_n^2$ is cofinal in any left ordering of $B_n$ \cite{DDRW08}, from which is follows by Proposition \ref{prop:cof} that $LO(B_n)$ contains no dense orbit.

\textbf{Acknowledgments.} The author would like to thank Crist\'{o}bal Rivas and Dale Rolfsen for much help with earlier drafts of the manuscript.

\bibliographystyle{plain}
\bibliography{candidacy}

\end{document}